\documentclass[parskip=half,toc=bib]{scrartcl}
\usepackage[utf8]{inputenc}
\usepackage{latexsym,amssymb,yfonts} %
\usepackage{amsmath}
\usepackage{amsthm}
\usepackage{mathtools}
\usepackage{tikz}
\usepackage{enumerate}
\usepackage{psfrag}
\usepackage{graphicx}
\usepackage[colorlinks=true,linkcolor=black,urlcolor=black,citecolor=black,bookmarks]{hyperref}
\usepackage{multirow}
\usepackage{booktabs}
\usepackage[style=numeric, sorting=nyt,maxnames=4,isbn=false,url=false,doi=false]{biblatex}
\usepackage[all]{xy}
\usepackage{caption}
\usepackage{thmtools}
\def\ResetSectionNames{%
  \def\chapterautorefname{Chapter}
  \def\sectionautorefname{Section}
  \def\subsectionautorefname{Section}
  \def\subsubsectionautorefname{Section}
  \def\figureautorefname{Figure}
  \def\indexname{List of symbols}
}

\AtBeginDocument\ResetSectionNames

\declaretheoremstyle[spaceabove=\topsep,spacebelow=0pt,bodyfont=\normalfont]{scdef}
\declaretheoremstyle[spaceabove=\topsep,spacebelow=0pt,bodyfont=\itshape]{scthm}
\declaretheoremstyle[spaceabove=\topsep,spacebelow=0pt,headfont=\normalfont\itshape,notefont=\normalfont\itshape,notebraces={}{},headformat={\NAME\NOTE},postheadspace=1em,qed=\qedsymbol]{scprf}

\declaretheorem[style=scthm,numberwithin=section,name=Theorem,    refname={Theorem,Theorems},        Refname={Theorem,Theorems}]        {theorem}
\declaretheorem[style=scthm,unnumbered ,name=Theorem,    refname={Theorem,Theorems},        Refname={Theorem,Theorems}]        {theorem*}
\declaretheorem[style=scthm,sharenumber=theorem,     name=Lemma,      refname={Lemma,Lemmas},            Refname={Lemma,Lemmas}]            {lemma}
\declaretheorem[style=scthm,sharenumber=theorem,     name=Corollary,  refname={Corollary,Corollaries},   Refname={Corollary,Corollaries}]   {corollary}
\declaretheorem[style=scthm,sharenumber=theorem,     name=Proposition,refname={Proposition,Propositions},Refname={Proposition,Propositions}]{prop}
\declaretheorem[style=scdef,sharenumber=theorem,     name=Definition, refname={Definition,Definitions},  Refname={Definition,Definitions}]  {definition}
\declaretheorem[style=scdef,sharenumber=theorem,     name=Remark,     refname={Remark,Remarks},          Refname={Remark,Remarks}]          {rem}

\renewbibmacro{in:}{}
\numberwithin{table}{section}
\newcommand{\cyclic}[1]{\stackrel{\scriptstyle #1}{\mathfrak{S}}}
\newcommand{\R}{\mathbb{R}}

\renewcommand{\H}{\mathcal{H}}

\newcommand\extalg{%
  \newlength{\len}%
  \settoheight{\len}{V}%
  \mathbin{%
    \resizebox{0.93\len}{0.93\len}{$\wedge$}%
    \kern-0.1em%
  }}%
\newcommand{\intprod}{\mathbin{\hbox to 0.7ex{%
      \kern-0.3ex
      \vrule height0.0777ex width0.971ex depth0ex
      \kern-0.055ex
      \vrule height1.165ex width0.0777ex depth0ex\hss}}%
}%

\AtEveryBibitem{\clearlist{language}}
\begin{document}

\title{Canonical Submersions in Nearly Kähler Geometry}

\author{Leander Stecker}
\date{}
\maketitle

\begin{abstract}
	We explore submersions introduced by reducible holonomy representations of connections with parallel skew torsion. A submersion theorem extending previous, less general, results is given. As our main application we show that parallel $3$-$(\alpha,\delta)$-Sasaki manifolds admit $1$-dimensional submersions onto nearly Kähler orbifolds. As a secondary application we reprove that a given class of nearly Kähler manifolds submerges onto quaternionic Kähler manifolds. This new proof gives a direct expression for the quaternionic structure on the base.
\end{abstract}
\medskip



\section{Introduction}
The investigation of holonomy groups has played a key role in Riemannian geometry over the last century. Thanks to de Rham splitting one can reduce the investigation to irreducible holonomy representations. It therefore is possible to classify them in the shape of Berger's list of special holonomies. While these geometries are interesting and widely studied, they are narrow and many, in particular odd dimensional, geometries fail to be represented.
An early generalization due to Gray were weak holonomy groups, extending to the classes of nearly Kähler, nearly parallel $G_2$ and nearly parallel $\mathrm{Spin}(9)$ manifolds \cite{GrayWeak, FriWeakSpin9}. We are foremost interested in the broader class of geometries that admit a metric connection with skew symmetric torsion, or geometries with skew torsion for short. These include the aforementioned classes of weak holonomies and are also of widespread interest in type II string theory, compare \cite{GMWSkew, FrIv, FriSpin9Skew}.

Unlike with the Levi-Civita connection, for connections with skew torsion there exist no general de Rham splitting so reducible representations are very much to be considered. In these cases the reducibility tells us a great deal about the geometry. In fact, under the assumption of parallel torsion R.~Cleyton, A.~Moroianu and U.~Semmelmann in \cite{CleyMorSemm} show that they always admit a locally defined Riemannian submersion. We refine their theorem giving us control over the submersion motivated by the case of $3$-$(\alpha,\delta)$-Sasaki manifolds in \cite{ADS20}. We will call them canonical submersions.

$3$-$(\alpha,\delta)$-Sasaki manifolds were defined by the condition $\mathrm{d}\eta_i =2\alpha\Phi_i + 2(\alpha-\delta) \xi_j\wedge\xi_k$ in \cite{AgrDil} as a generalization of $3$-Sasaki manifolds when $\alpha=\delta=1$. Pivotally they admit a canonical connection $\nabla$ with skew torsion.  In \cite{ADS20} it was shown that $\nabla$ gives rise to a canonical submersion in above sense, whose base admits a quaternionic Kähler, resp.~hyperkähler, structure. This divides them into positive, negative and degenerate $3$-$(\alpha,\delta)$-Sasaki structures, depending on the scalar curvature of the quaternionic Kähler base. In the positive realm there are various interesting cases of parameters. Apart from aforementioned $3$-Sasakian if $\alpha=\delta=1$ there is the secondary Einstein metric or so-called squashed $3$-Sasaki Einstein metric if $\delta=(2n+3)\alpha$, \cite{AgrDil}. We are particularly interested in the parallel case $\delta=2\alpha$. Here $\nabla$ parallelizes all Reeb vector fields. We will show that this yields apart from the submersion discussed in \cite{CleyMorSemm} and \cite{ADS20}, another admissible submersion of codimension $1$ onto nearly Kähler spaces. This is the first instance investigated of canonical submersions whose base is a geometry with non-vanishing torsion showing that this possibility is not an inconvenience, but a feature of the theory.

Nearly Kähler manifolds are almost Hermitian, non-Kähler manifolds, defined by the condition $(\nabla_X J)X=0$. They garnered particular attention in dimension $6$. In this dimension their cone has holonomy $G_2$, \cite{Bry}, they admit a real Killing spinor and are Einstein. Though this is not true for dimensions greater $6$, they are connected to many geometries in any dimension. For instance, they are known to exist on the twistor spaces of quaternionic Kähler manifolds, \cite{EellsSalamon, FriGru}.  It therefore shouldn't come as a surprise they can be obtained as a quotient of the Konishi space. However, our proof shows that the nearly Kähler structure and its canonical Hermitian connection are tightly locked with the parallel $3$-$(\alpha,\delta)$-Sasaki structure and its canonical connection. A second application closes the circle showing how a class of nearly Kähler manifolds, essentially those obtained by above construction, admit a further canonical submersion onto quaternionic Kähler spaces. This result was shown previously in the decomposition of nearly Kähler spaces of $\dim\geq 6$ by P.-A.~Nagy in \cite{Nagy} and locally in \cite{Alexandrov}. However, we include it since the proof fits neatly into our established framework of canonical submersions.

\section{Canonical Submersions}

On a Riemannian manifold $(M,g)$ a metric connection $\nabla$ on the tangent bundle is uniquely characterized by its torsion tensor
\begin{align*}
T(X,Y,Z)\coloneqq g(\nabla_XY-\nabla_YX-[X,Y],Z).
\end{align*}
By definition the torsion is skew-symmetric in the first two entries. If, however it is totally skew-symmetric we say that $\nabla$ has \emph{skew torsion}. In this case the connection is given by
\begin{align*}
g(\nabla_XY,Z)=g(\nabla^g_XY,Z)+\frac 12 T(X,Y,Z)
\end{align*}
where we denote by $\nabla^g$ the Levi-Civita connection of $(M,g)$. We require additionally that our connection $\nabla$ has parallel torsion, i.e.~$\nabla T=0$. This yields many simplifications, among them such connections satisfy the following Bianchi identity:
\begin{align}\label{Bianchi}
\cyclic{X,Y,Z}R^\nabla(X,Y,Z,V)=\sigma_T(X,Y,Z,V)\coloneqq\cyclic{X,Y,Z}g(T(X,Y),T(Z,V)),
\end{align}
where the shorthand notation $\cyclic{X,Y,Z}$ denotes the sum over all cyclic permutations of $X,Y,Z$.

With those preliminaries we pose our first main theorem. Given a manifold admitting a connection with parallel skew torsion and reducible holonomy then the Canonical Submersion Theorem answers the question how such a manifold can be simplified. Earlier more restrictive version of this theorem were presented in \cite{CleyMorSemm}, \cite{ADS20} and \cite{mythesis}. However, since the assumptions are simplified we include the proof.

\begin{theorem}\label{canonicalsub}
Suppose $\nabla$ is a metric connection with parallel skew torsion $T$ on $(M,g)$ and $TM=\mathcal{H}\oplus\mathcal{V}$ splits orthogonally as representation of the reduced holonomy group $\mathrm{Hol}_0(\nabla)$. Assume further that 
\begin{align}\label{projecttau}
T\in \Lambda^3\mathcal{H}\oplus\Lambda^2\mathcal{H}\!\wedge\!\mathcal{V}\oplus \Lambda^3\mathcal{V}\subsetneq \Lambda^3TM,
\end{align}
i.e.\ the $\Lambda^2\mathcal{V}\wedge\mathcal{H}$-part of $T$ vanishes.

Then there exists a locally defined Riemannian submersion $\pi\colon (M,g)\to (N,g_N)$ with totally geodesic fibers tangent to $\mathcal{V}$, the purely horizontal part of the torsion $T^\mathcal{H}$ is projectable, $\pi^*\check{T}\coloneqq T^\mathcal{H}$, and $\nabla^{\check{T}}=\nabla^{g_N}+\frac 12 \check{T}$ is a connection with parallel skew torsion on $N$ satisfying
\begin{align}\label{pinabla}
\nabla^{\check{T}}_XY=\pi_*(\nabla_{\overline{X}}\overline{Y}),
\end{align}
where $\overline{X},\overline{Y}$ denote the horizontal lifts of vector fields $X,Y\in \Gamma TN$.
\end{theorem}

\begin{proof}
We note that by \eqref{projecttau} and the invariance of $\mathcal{V}$ under $\nabla$ for any vertical vector fields $V,W\in \mathcal{V}$ we have
\begin{align*}
\nabla^g_VW=\nabla_VW-\frac 12T(V,W)\in\mathcal{V}.
\end{align*}
Therefore, the distribution $\mathcal{V}$ is integrable and a curve on the integral submanifold tangent to $\mathcal{V}$ is a geodesic if and only if it is a geodesic in $M$. The integral submanifolds give rise to a foliation and hence to a submersion $\pi$ from a small neighborhood $U\subset M$ to a local transverse section $S$. We show that the metric restricted to $\mathcal{H}\times\mathcal{H}$ is constant along vertical vector fields and therefore projectable. For $X,Y\in\Gamma\H$ we have
\begin{align*}
(L_Vg)(X,Y)&=V(g(X,Y))-g([V,X],Y)-g(X,[V,Y])=g(\nabla^g_XV,Y)+g(X,\nabla^g_YV)\\
&=g(\nabla_XV,Y)+g(X,\nabla_YV)=0
\end{align*}
since $\mathcal{V}$ is preserved by $\nabla$. This proves that $\pi$ is a Riemannian submersion.

To prove the second assertion we denote $T^\mathcal{H}=\mathrm{pr}_{\Lambda^3\mathcal{H}}T$. If we show that $T^\mathcal{H}$ is constant along the fibers it projects to a well-defined $3$-form $\check{T}$. Let $V\in \mathcal{V}$. Then
\begin{align}\label{Torsionintorsion}
\cyclic{X,Y,Z}T^\mathcal{H}(X,Y,T(V,Z))=0
\end{align}
whenever either $X,Y,Z\in \mathcal{V}$. Indeed, by \eqref{projecttau} we have that $T(V,Z)\in\mathcal{V}$ if $Z\in \mathcal{V}$ and, hence, $T(V,Z)\intprod T^\mathcal{H}=0$. Now consider $X,Y,Z\in\mathcal{H}$. Since $\nabla$ preserves $\mathcal{V}$ and $\mathcal{H}$ the curvature $R^{\nabla}(X,Y,Z,V)=0$ and the Bianchi identity for connections with parallel skew torsion, \eqref{Bianchi}, implies
\begin{align*}
0&=-\!\!\cyclic{X,Y,Z} R^{\nabla}(X,Y,Z,V)=-\sigma_T(X,Y,Z,V)=-\!\!\cyclic{X,Y,Z}g(T(X,Y),T(Z,V))\\
&=\cyclic{X,Y,Z}T(X,Y,T(V,Z))=\cyclic{X,Y,Z}T^\mathcal{H}(X,Y,T(V,Z)),
\end{align*}
where the last step used again that $T(V,Z)\in \mathcal{H}$. Thus, \eqref{Torsionintorsion} holds for any $X,Y,Z\in TM$.

As the subspaces $\mathcal{V}$, $\mathcal{H}$ are preserved by $\nabla$ so are the components of tensors on $TM$. In particular, $\nabla T=0$ implies $\nabla T^{\mathcal{H}}=0$. Use \eqref{Torsionintorsion} and $\nabla_ZV\in \mathcal{V}$ to obtain
\begin{align*}
(L_VT^\mathcal{H})(X,Y,Z)&=V(T^\mathcal{H}(X,Y,Z))-\!\!\cyclic{X,Y,Z}T^\mathcal{H}(X,Y,L_VZ)\\
&=(\nabla_VT^{\mathcal{H}})(X,Y,Z)+\!\!\cyclic{X,Y,Z}T^{\mathcal{H}}(X,Y,\nabla_VZ)\\
&\qquad\qquad-\!\!\cyclic{X,Y,Z}T^\mathcal{H}(X,Y,\nabla_VZ-\nabla_ZV-T(V,Z))\\
&=\cyclic{X,Y,Z}T^\mathcal{H}(X,Y,T(V,Z))=0.
\end{align*}
Equation \eqref{pinabla} follows directly from $\nabla^{g_N}_{X}{Y}=\pi_*(\nabla^{g}_{\overline{X}}\overline{Y})$ for Riemannian submersions. 
Finally we conclude that $\check{T}$ is $\nabla^{\check{T}}$-parallel:
\begin{align*}
(\nabla^{\check{T}}_A\check{T})(X,Y,Z)&=A(\check{T}(X,Y,Z))-\!\!\cyclic{X,Y,Z}\check{T}(X,Y,\nabla_A^{\check{T}} Z)\\
&=A(\pi\circ T^\mathcal{H}(\overline{X},\overline{Y},\overline{Z}))-\!\!\cyclic{X,Y,Z}\check{T}(X,Y,\pi_*(\nabla_{\overline{A}}\overline{Z}))\\
&=\overline{A}(T^\mathcal{H}(\overline{X},\overline{Y},\overline{Z}))-\!\!\cyclic{X,Y,Z}T^\mathcal{H}(\overline{X},\overline{Y},\nabla^T_{\overline{A}}\overline{Z})\\
&=(\nabla_{\overline{A}}T^\mathcal{H})(\overline{X},\overline{Y},\overline{Z})=0.\qedhere
\end{align*}
\end{proof}

Observe that compared to earlier versions this yields the following well known generalization of de Rham splitting as an immediate consequence.

\begin{corollary}
Let $\nabla$ be a connection with parallel skew torsion $T$ and $TM=V_1\oplus V_2$ a holonomy-invariant decomposition. Then locally
\begin{align*}
(M,g,\nabla)=(M_1,g_1,\nabla^{T_1})\times (M_2,g_2,\nabla^{T_2})
\end{align*}
if and only if the torsion is decomposable, i.e. $T=T_1+T_2$ with $T_i\in\Lambda^3V_i$.
\end{corollary}

\begin{proof}
Here \eqref{projecttau} is satisfied for $\mathcal{V}=V_1$, $\mathcal{H}=V_2$ and vice versa. Hence, both distributions are integrable and we obtain locally defined Riemannian projection maps to their respective integral submanifolds. The converse is clear.
\end{proof}

\begin{rem}
A splitting $TM=\mathcal{V}\oplus\mathcal{H}$ satisfying \eqref{projecttau} is not necessarily unique. In \cite{CleyMorSemm} it is shown that, if there exists any $\mathrm{Hol}_0(\nabla)$-invariant splitting, there exists one with maximal $\mathcal{V}$ satisfying \eqref{projecttau}. This was recently coined as the canonical splitting in \cite{MorSch} compared to other admissible splittings satisfying \eqref{projecttau}. Our main application in \autoref{parallelovernk} reflects that distinction. It is crucially an admissible, but never canonical, splitting.
\end{rem}

We now show that a further splitting of the holonomy representation can manifest a splitting on the base.

\begin{prop}\label{holbelow}
Suppose the holonomy representation of $\nabla$ splits into orthogonal submodules $TM=\mathcal{V}\oplus\mathcal{H}_1\oplus\mathcal{H}_2$ such that $\mathcal{V}$ and $\mathcal{H}=\mathcal{H}_1\oplus\mathcal{H}_2$ satisfy condition \eqref{projecttau} of \autoref{canonicalsub}. Let $\pi\colon M\to N$ denote the canonical submersion and suppose further that $\mathcal{H}_1$ and $\mathcal{H}_2$ are projectable. Then the representation $TN$ of the holonomy $\mathrm{Hol}_0(\nabla^{\check{T}})$ is reducible into modules $\pi_*\mathcal{H}_1\oplus\pi_*\mathcal{H}_2$. 
\end{prop}

\begin{proof}
By the Ambrose-Singer theorem the holonomy algebra at $x\in N$ is generated by elements of the form
\begin{align*}
(\mathcal{P}^{\nabla^{\check{T}}}_\gamma)^{-1}\circ\mathcal{R}^{\nabla^{\check{T}}}(X\wedge Y)\circ \mathcal{P}_\gamma^{\nabla^{\check{T}}}
\end{align*}
where $\mathcal{P}^{\nabla^{\check{T}}}_\gamma$ denotes parallel transport along a piecewise smooth curve $\gamma$ from $x$ to $p$ and $\mathcal{R}^{\nabla^{\check{T}}}(X\wedge Y)\in \Lambda^2T_pN\subset \mathrm{End}(T_pN)$ is the evaluation of the curvature operator at $X,Y\in T_pN$. For any such curve let $\overline{\gamma}$ be the horizontal lift starting at a point $x_0\in \pi^{-1}(x)$. Then let $X(t)$ be the unique parallel vector field along $\gamma$. By \eqref{pinabla}
\begin{align*}
\pi_*(\nabla_{\dot{\overline{\gamma}}}\overline{X(t)})=\nabla^{\check{T}}_{\dot{\gamma}}X(t)=0
\end{align*}
and hence $\nabla_{\dot{\overline{\gamma}}}\overline{X(t)}\in\mathcal{V}$. However, $\nabla_{\dot{\overline{\gamma}}} \overline{X(t)}\in \mathcal{H}$ since $\overline{X(t)}\in\Gamma\mathcal{H}$ and $\mathcal{H}$ is invariant under the holonomy of $\nabla$. Therefore $\nabla_{\dot{\overline{\gamma}}} \overline{X(t)}=0$ and
\begin{align*}
\mathcal{P}^{\nabla}_{\overline{\gamma}}\overline{X}=\overline{\mathcal{P}^{\nabla^{\check{T}}}_\gamma X}.
\end{align*}
In particular, parallel transport with respect to $\nabla^{\check{T}}$ preserves $\pi_*\mathcal{H}_1$ and $\pi_*\mathcal{H}_2$. Now let $X,Y\in T_pN$, then
\begin{align*}
R^{\nabla^{\check{T}}}(X\wedge Y)Z&=\nabla^{\check{T}}_X\nabla^{\check{T}}_YZ-\nabla^{\check{T}}_Y\nabla^{\check{T}}_XZ-\nabla^{\check{T}}_{[X,Y]}Z\\
&=\pi_*(\nabla_{\overline{X}}\nabla_{\overline{Y}}\overline{Z}-\nabla_{\overline{Y}}\nabla_{\overline{X}}\overline{Z}-\nabla_{\overline{[X,Y]}}\overline{Z})
\end{align*}
preserves the modules $\pi_*\mathcal{H}_1$ and $\pi_*\mathcal{H}_2$ as well.
\end{proof}

The following computational lemma has already been used in \cite{ADS20}. However, it holds for any canonical submersion.

\begin{lemma}\label{nablavert}
Let $\nabla$ be a connection as in \autoref{canonicalsub}. Then
\begin{align*}
g(\nabla_XY,Z)=T(X,Y,Z)
\end{align*}
for any vertical vector $X\in\mathcal{V}$, horizontal vector $Z\in \mathcal{H}$, and basic vector field $Y\in\Gamma\mathcal{H}$. In particular, the expression is tensorial.
\end{lemma}

\begin{proof}
Recall that a vector field $Y$ is basic if it is horizontal and $\pi$-related to a vector field on $N$. In particular, we have $\pi_*[X,Y]=[\pi_*X,\pi_*Y]=0$ for any vector field $X\in\Gamma \mathcal{V}$. Note also that $\nabla_YX\in \mathcal{V}$ as the decomposition is $\mathfrak{hol}(\nabla)$-invariant. Therefore,
\begin{align*}
g(\nabla_XY,Z)&=g(\nabla_YX,Z)+g([X,Y],Z)+T(X,Y,Z)=T(X,Y,Z).\qedhere
\end{align*}
\end{proof}

\section{Parallel \texorpdfstring{$3$-$(\alpha,\delta)$}{3-(a,d)}-Sasaki and Nearly Kähler Manifolds}
We like to introduce the structures involved in the application of canonical submersions we consider in this paper, in particular $3$-$(\alpha,\delta)$-Sasakian and nearly Kähler manifolds. An almost contact metric manifold $(M^{2n+1},g,\xi,\eta,\varphi)$ is given by an odd dimensional Riemannian manifold, a unit length vector field $\xi\in \Gamma TM$ called Reeb vector field, its metric dual $1$-form $\eta$ and an almost Hermitian structure $\varphi\in \Gamma \mathrm{End}(TM)$ on $\ker \eta$ satisfying
\begin{gather*}
\varphi\xi=0,\qquad \eta\circ\varphi=0,\qquad \varphi^2=-\mathrm{id}+\xi\otimes\eta,\\
g(\varphi X,\varphi Y)=g(X,Y)-\eta(X)\eta(Y).
\end{gather*}
These are considered the odd dimensional analogue of almost Hermitian manifolds. As for the latter we define the \emph{fundamental} $2$-form $\Phi(X,Y)=g(X,\varphi Y)$ for a given almost contact metric manifold.

A tuple $(M,g,\xi_i,\varphi_i,\eta_i)_{i=1,2,3}$ of three almost contact metric structures on the same underlying Riemannian manifold is called almost $3$-contact metric manifold if they additionally satisfy the compatibility conditions
\begin{align*}
\varphi_i\xi_j=\xi_k,\qquad \eta_i\circ\varphi_j=\eta_k,\qquad \varphi_i\varphi_j=\varphi_k+\xi_i\otimes\eta_j.
\end{align*}
These properties guarantee that the endomorphisms act as imaginary quaternions on the \emph{horizontal} distribution $\mathcal{H}\coloneqq \bigcap \ker\eta_i$. Accordingly the complementary distribution $\mathcal{V}\coloneqq \mathcal{H}^\perp=\mathrm{span}\{\xi_1,\xi_2,\xi_3\}$ is called \emph{vertical}.

\begin{definition}[\cite{AgrDil}]\label{3ad}
An almost $3$-contact metric manifold is called \emph{$3$-$(\alpha,\delta)$-Sasaki manifold} for real constants $\delta$ and $\alpha\neq 0$ if
\begin{align}\label{def3ad}
d\eta_i=2\alpha\Phi_i+2(\alpha-\delta)\eta_j\wedge \eta_k
\end{align}
for every even permutation $(ijk)$ of $(123)$.
\end{definition}

Most prominently, if $\alpha=\delta=1$ the manifold is $3$-Sasakian. A second class often considered is the second Einstein metric with parameters $\delta=(2n+3)\alpha$ where $\dim M=4n+3$. Our work highlights the subclass of \emph{parallel $3$-$(\alpha,\delta)$-Sasaki manifolds}, that is if $\delta=2\alpha$. Their most prominent feature is highlighted with regards to the canonical connection as defined via the following theorem.

\begin{theorem}[\cite{AgrDil}]\label{canconn}
A $3$-$(\alpha,\delta)$-Sasaki manifold admits a unique metric connection $\nabla$ with skew torsion such that
\begin{align}\label{eq:canconn}
\nabla_X\varphi_i&=\beta(\eta_k(X)\varphi_j-\eta_j(X)\varphi_k),
\end{align}
for every even permutation $(ijk)$ of $(123)$ and with $\beta=2(\delta-2\alpha)$.

The connection $\nabla$ preserves the splitting $TM=\mathcal{V}\oplus\mathcal{H}$ and its torsion is given by
\begin{align}\label{cantorsion}
T=2\alpha\sum_{i=1}^3\eta_i\wedge\Phi_i-2(\alpha-\delta)\eta_{123}=2\alpha\sum_{i=1}^3\eta_i\wedge\Phi_i^\mathcal{H}+2(\delta-4\alpha)\eta_{123},
\end{align}
where $\Phi_i^\mathcal{H}\coloneqq \Phi_i|_\mathcal{H}$. In particular, the torsion is parallel $\nabla T=0$.
\end{theorem}

Hence for parallel $3$-$(\alpha,\delta)$-Sasaki manifolds the canonical connection parallelizes all structure tensors. On homogeneous parallel $3$-$(\alpha,\delta)$-Sasaki manifolds it was shown that the metric is naturally reductive with respect to the standard presentation as a quotient of the automorphism group, see \cite{ADS20}. In that case the canonical and Ambrose-Singer connections coincide.

The canonical connection also gave rise for an initial instance of canonical submersions over quaternionic Kähler spaces in \cite{ADS20}. Recall that \emph{quaternionic Kähler} manifolds $(\check{N}^{4n},g_{\check{N}},\mathcal{Q})$, $n\geq 2$ are equipped with a $3$-dimensional subbundle $\mathcal{Q}\subset\mathrm{End}(T\check{N})$ locally generated by a triple of almost Hermitian structures that is invariant under $\mathrm{Hol}(\nabla^{g_{\check{N}}})$. Equivalently, they are the class of Riemannian manifolds with exceptional holonomy given by a subgroup of $\mathrm{Sp}(1)\mathrm{Sp}(n)$. For $\dim \check{N}=4$ we require $\check{N}$ to be Einstein and anti-self-dual.

\begin{theorem}[\cite{ADS20}]
Let $(M,g,\xi_i,\varphi_i,\eta_i)_{i=1,2,3}$ be a $3$-$(\alpha,\delta)$-Sasaki manifold. Then there exists a locally defined Riemannian submersion $\pi\colon M\to \check{N}$ with totally geodesic fibers tangent to $\mathcal{V}$ such that $\nabla^{g_{\check{N}}}_XY=\pi_*(\nabla_{\overline{X}}\overline{Y})$. The base space $\check{N}$ admits a quaternionic Kähler structure locally generated by the $3$ almost Hermitian structures
\begin{align*}
I_i=\pi_*\circ \varphi_i\circ s_*,
\end{align*}
for any locally defined section $s\colon \check{N}\to M$.
\end{theorem}

We should remark that in the following we will fix $\mathcal{H}$ and $\mathcal{V}$ to be the horizontal and vertical distributions for a $3$-$(\alpha,\delta)$-Sasaki manifold while the horizontal/vertical subspaces in further applications of \autoref{canonicalsub} will be denoted appropriately to the situation.

For later use we recall the following Lie derivatives along vertical vectors.

\begin{prop}[\cite{AgrDil}]
For any $3$-$(\alpha,\delta)$-Sasaki manifold and any even permutation $(ijk)$ of $(123)$ we have the following identities
\begin{gather}
\mathcal{L}_{\xi_i}\varphi_i=0,\qquad \mathcal{L}_{\xi_i}\varphi_j=-\mathcal{L}_{\xi_j}\varphi_i=2\delta\varphi_k\label{Lvarphi}\\
\mathcal{L}_{\xi_i}\xi_i=0,\, \qquad \mathcal{L}_{\xi_i}\xi_j=\,-\mathcal{L}_{\xi_j}\xi_i=2\delta\xi_k\label{liexi}
\end{gather}
\end{prop}
With that let us shift attention to nearly Kähler manifolds.
\begin{definition}
A nearly Kähler manifold $(N^{2m},g_N,J)$ is an almost Hermitian manifold such that $(\nabla^{g_N}_XJ)X=0$ for all $X\in TN$.
\end{definition}
%

As initially observed in \cite{GraynK} a nearly Kähler manifold admits a particularly nice connection $\nabla^c$. We will call $\nabla^c$ the \emph{characteristic} connection of a nearly Kähler manifold as it is the unique connection with skew torsion preserving the $U(m)$-structure, compare \cite{Ag06}. In the literature it is often also called Bismut or canonical connection. In fact their usual definitions agree on nearly Kähler manifolds as it is the unique Hermitian connection and has torsion 
\begin{align*}
T^c(X,Y,Z)=g_N((\nabla^{g_N}_XJ)JY,Z).
\end{align*}
Furthermore, $T^c$ is parallel with respect to $\nabla^c$.

\section{Nearly Kähler orbifolds from parallel \texorpdfstring{$3$-$(\alpha,\delta)$}{3-(a,d)}-Sasaki manifolds}

Let $(M,g,\xi_i,\varphi_i,\eta_i)_{i=1,2,3}$ be a $3$-$(\alpha,\delta)$-Sasaki manifold and $\nabla$ its canonical connection with parallel skew torsion $T$. Choose any almost contact structure $\varphi$ in the associated sphere $\Sigma=\{a_1\varphi_1+a_2\varphi_2+a_3\varphi_3\ \vert\  |a|=1\}$ and denote its corresponding Reeb vector field $\xi$ and contact form $\eta$. If $\beta=0$, or equivalently $\delta=2\alpha$, we have $\nabla\varphi=0$ and also $\nabla \xi=0$ and $\nabla \eta=0$. Hence, the decomposition $TM=\langle\xi\rangle\oplus\ker\eta=\langle\xi\rangle\oplus\langle\xi\rangle^\perp$ is holonomy-invariant. Note that \eqref{projecttau} is trivially satisfied for any one dimensional vertical distribution. In the following we apply \autoref{canonicalsub} to this setup. We fix the notation $\mathcal{V}$ and $\mathcal{H}$ for the vertical and horizontal spaces of the initial parallel $3$-$(\alpha,\delta)$-manifold and denote the vertical and horizontal spaces with respect to this newly obtained submersion by $\langle\xi\rangle$ and $\langle\xi\rangle^\perp$ respectively.

\begin{theorem}\label{parallelovernk}
Let $(M,g,\xi_i,\varphi_i,\eta_i)_{i=1,2,3}$ be a parallel $3$-$(\alpha,\delta)$-Sasaki manifold and fix an almost contact metric structure $(\xi,\varphi,\eta)$ inside the associated sphere.
\begin{enumerate}[a)]
\item Then there exists a locally defined Riemannian submersion $\pi\colon (M,g)\to (N,g_N)$ along the orbits of $\xi$.
\item Set $\tilde{\varphi}\coloneqq \varphi|_{\mathcal{H}}-\varphi|_{\mathcal{V}}$ and $J=\pi_*\circ \tilde{\varphi}\circ s_*$ with an arbitrary local section $s\colon N\to M$ of $\pi$. Then $(N,g_N,J)$ is nearly Kähler.
\item The characteristic connection on $(N,g_N,J)$ agrees with the connection obtained from the canonical connection on $M$. In particular, $\check{T}(X,Y,Z)=g_N((\nabla^{g_N}_XJ)JY,Z)$.
\end{enumerate}
\end{theorem}

\begin{proof}
We have already seen that the assumptions in \autoref{canonicalsub} are satisfied. We may in the following assume $\xi=\xi_1$. \autoref{canonicalsub} then implies that there is a locally defined Riemannian submersion $\pi\colon (M,g)\to (N,g_N)$ along the orbits of $\xi_1$ and horizontal space $\langle\xi_1\rangle^\perp=\langle\xi_2,\xi_3\rangle\oplus\mathcal{H}$. Further, there is a connection $\nabla^{\check{T}}=\nabla^{g_N}+\frac 12 \check{T}$ on $TN$ with parallel skew torsion given by
\begin{align}\label{checkT}
\pi^*\check{T}=T^{\langle\xi\rangle^\perp}=2\alpha(\eta_2\wedge\Phi^\mathcal{H}_2+\eta_3\wedge\Phi^\mathcal{H}_3)
\end{align}
where we have used \eqref{cantorsion}. This connection satisfies
\begin{align*}
\nabla^{\check{T}}_XY=\pi_*(\nabla_{\overline{X}} \overline{Y}).
\end{align*}
Now consider $J=\pi_*\circ \tilde{\varphi}\circ s_*$. Due to \eqref{Lvarphi} we have $\mathcal{L}_\xi\varphi=0$ and since $\mathcal{L}_\xi$ preserves $\mathcal{H}$ and $\mathcal{V}$, also $\mathcal{L}_\xi\tilde{\varphi}=0$. In particular, $J$ is independent of the choice of $s$. The compatibility with $g_N$ follows immediately as $\pi_*\colon \langle\xi\rangle^\perp\to TN$ and $\mathrm{pr}_{\langle\xi\rangle^\perp}\circ s_*\colon TN\to\langle\xi\rangle^\perp$ are isometric. We check that $J^2=-\operatorname{id}$. Since $s$ is a section of $\pi$ we have $s_*\circ \pi_*=\operatorname{id}$ on the image of $s$ and thus
\[
J^2=\pi_*\circ\tilde{\varphi}\circ s_*\circ \pi_*\circ \tilde{\varphi}\circ s_*=\pi_*\circ\tilde{\varphi}^2\circ s_*=\pi_*\circ(-\operatorname{id}+\eta\otimes\xi)\circ s_*=-\operatorname{id}
\]
where we have used that all involved endomorphisms preserve the orthogonal splitting $\mathcal{H}\oplus\langle\xi_2,\xi_3\rangle$.

We check that $J$ is parallel with respect to $\nabla^{\check{T}}$. Remark that the horizontal lift of any vector field on $TN$ is basic and $\overline{\pi_*(\tilde{\varphi}_1(s_*Y))}=(\tilde{\varphi}_1(s_*Y))_{\mathcal{H}\oplus\langle\xi_2,\xi_3\rangle}$ wherever  the right side is defined, that is on the image $s(N)$. Set $\hat{X}\coloneqq \overline{X}-s_*X$ the vertical part of $-s_*X$. Then
\begin{align*}
(\nabla^{\check{T}}_X J)Y&=\nabla^{\check{T}}_X (JY) - J(\nabla^{\check{T}}_XY)=\pi_*(\nabla_{\overline{X}} \overline{JY}) - J(\pi_*(\nabla_{\overline{X}} \overline{Y}))\\
&=\pi_*(\nabla_{\overline{X}} (\overline{\pi_*(\tilde{\varphi}_1 (s_*Y))} ) - \tilde{\varphi}_1(s_*(\pi_* (\nabla_{\overline{X}} \overline{Y}))))\\
&= \pi_*(\nabla_{\hat{X}} (\overline{\pi_*(\tilde{\varphi}_1 (s_*Y))} ) + \nabla_{s_*X} (\overline{\pi_*(\tilde{\varphi}_1 (s_*Y))} ) - \tilde{\varphi}_1(\nabla_{\hat{X}} \overline{Y})- \tilde{\varphi}_1(\nabla_{s_*X} \overline{Y})).
\end{align*}
Remark that the horizontal lift of any vector field on $TN$ is basic so we may employ \autoref{nablavert}. In our case
\begin{align*}
g(\nabla_{\xi_1}H,Z)&=T(\xi_1,H,Z)=2\alpha\left(\sum_{i=1}^3\eta_i\wedge\Phi_i^\mathcal{H}-2\eta_{123}\right)(\xi_1,H,Z)\\
&=2\alpha(\Phi_1^\mathcal{H}(H,Z)-2\eta_{23}(H,Z)),
\end{align*}
where $H$ is either $\overline{\pi_*(\tilde{\varphi}_1 (s_*Y))}$ or $\overline{Y}$. Note that we apply $\pi_*$ in the end so it suffices to assume $Z\in\langle\xi_1\rangle^\perp$ in the following:
\begin{align*}
g(\nabla_{\hat{X}} (\overline{\pi_*(\tilde{\varphi}_1 (s_*Y))},Z)&=2\alpha\eta_1(\hat{X})(\Phi_1^\mathcal{H}(\overline{\pi_*(\tilde{\varphi}_1 (s_*Y))},Z) - 2\eta_{23}(\overline{\pi_*(\tilde{\varphi}_1 (s_*Y))},Z))\\
&=2\alpha\eta_1(\hat{X})(\Phi_1^\mathcal{H}(\tilde{\varphi}_1 (s_*Y),Z) - 2\eta_{23}(\tilde{\varphi}_1 (s_*Y),Z))\\
&=2\alpha\eta_1(\hat{X})(\Phi_1^\mathcal{H}(\varphi_1 (s_*Y),Z) + 2\eta_{23}(\varphi_1 (s_*Y),Z))\\
&=2\alpha\eta_1(\hat{X})(g((s_*Y)_\mathcal{H},Z) - 2(\eta_3(s_*Y)\eta_3(Z)+\eta_2(s_*Y)\eta_2(Z)))\\
&=2\alpha\eta_1(\hat{X})(g((s_*Y)_\mathcal{H},Z) - 2g((s_*Y)_{\langle\xi_2,\xi_3\rangle},Z))\\
g(\tilde{\varphi}_1(\nabla_{\hat{X}}\overline{Y}),Z)&=-2\alpha\eta_1(\hat{X})(\Phi_1^\mathcal{H}(\overline{Y},\tilde{\varphi}_1Z) - 2\eta_{23}(\overline{Y},\tilde{\varphi}_1Z))\\
&=-2\alpha\eta_1(\hat{X})(\Phi_1^\mathcal{H}(\overline{Y},\varphi_1Z) + 2\eta_{23}(\overline{Y},\varphi_1Z))\\
&=2\alpha\eta_1(\hat{X})(g(\overline{Y},Z) - 2(\eta_{2}(\overline{Y})\eta_2(Z)+\eta_{3}(\overline{Y})\eta_3(Z)))\\
&=2\alpha\eta_1(\hat{X})(g((\overline{Y})_\mathcal{H},Z) - 2g((\overline{Y})_{\langle\xi_2,\xi_3\rangle},Z))
\end{align*}
Since $(s_*Y)_{\mathcal{H}\oplus\langle\xi_2,\xi_3\rangle}=\overline{Y}$ both terms cancel. Further, both $\tilde{\varphi}_1$ and $\nabla$ preserve the splitting $TM=\R\xi_1\oplus\langle \xi_2,\xi_3\rangle\oplus\mathcal{H}$, thus $\nabla_{s_*X}s_*Y-\nabla_{s_*X}\overline{Y}$ is vertical and 
\begin{align}
(\nabla^{\check{T}}_X J)Y=\pi_*(\nabla_{s_*X} (\tilde{\varphi}_1 (s_*Y)) - \tilde{\varphi}_1(\nabla_{s_*X} s_*Y))=\pi_*((\nabla_{s_*X}\tilde{\varphi}_1)s_*Y)=0.
\end{align}
In order to control the covariant derivative of $J$ with respect to the Levi-Civita connection we need to compute 
\begin{align*}
g((\check{T}_X\cdot J)Y,Z)=g(\check{T}_X\cdot (JY),Z)-g(J(\check{T}_X\cdot Y),Z)=\check{T}(X,JY,Z)+\check{T}(X,Y,JZ).
\end{align*}
This is linear so we may compute it for any combination of $X,Y,Z$ in either $\pi_*\mathcal{H}$ or $\pi_*\mathcal{V}$ individually. Note that $J$ preserves this splitting and $\pi^*\check{T}\in \mathcal{V}\wedge\Lambda^2\mathcal{H}$ by \eqref{checkT}. Thus we only need to check for these combination of vectors. Further, note that $g((\check{T}_X\cdot J)Y,Z)$ is skew-symmetric in $Y,Z$. Two cases remain. For $X\in \pi_*\mathcal{V}$ and $Y,Z\in \pi_*\mathcal{H}$ we have
\begin{align}\label{TJ1}
\begin{split}
g((\check{T}_X\cdot J)Y,Z)&=\check{T}(X,JY,Z)+\check{T}(X,Y,JZ)\\
&=2\alpha\sum_{i=2,3}\eta_i(\overline{X})(\Phi_i(\varphi_1s_*Y,\overline{Z})+\Phi_i(\overline{Y},\varphi_1s_*Z))\\
&=2\alpha\sum_{i=2,3}\eta_i(\overline{X})(g(\overline{Y},\varphi_i\varphi_1\overline{Z})-g(\overline{Y},\varphi_1\varphi_i\overline{Z}))\\
&=-4\alpha(\eta_2(\overline{X})\Phi_3(\overline{Y},\overline{Z})-\eta_3(\overline{X})\Phi_2(\overline{Y},\overline{Z})).
\end{split}
\end{align}
For $X,Z\in\pi_*\mathcal{H}$ and $Y\in\pi_*\mathcal{V}$
\begin{align}\label{TJ2}
\begin{split}
g((\check{T}_X\cdot J)Y,Z)&=\check{T}(X,JY,Z)+\check{T}(X,Y,JZ)\\
&=-2\alpha\sum_{i=2,3}\left(\eta_i(\tilde{\varphi}_1s_*Y)\Phi_i(\overline{X},\overline{Z})+\eta_i(\overline{Y})\Phi_i(\overline{X},\varphi_1s_*Z)\right)\\
&=-2\alpha(\eta_3(\overline{Y})\Phi_2(\overline{X},\overline{Z})-\eta_2(\overline{Y})\Phi_3(\overline{X},\overline{Z})\\
&\qquad - \eta_2(\overline{Y})\Phi_3(\overline{X},\overline{Z}) + \eta_3(\overline{Y})\Phi_2(\overline{X},\overline{Z}))\\
&=4\alpha(\eta_2(\overline{Y})\Phi_3(\overline{X},\overline{Z})-\eta_3(\overline{Y})\Phi_2(\overline{X},\overline{Z})).
\end{split}
\end{align}
This implies that $(\nabla^{g_N}_XJ)X=(\nabla^{\check{T}}_XJ)X-\frac12 (\check{T}_X\cdot J)X=0$. Indeed, \eqref{TJ2} proves that the $\pi_*\mathcal{H}\times\pi_*\mathcal{H}$-part is skew, and the sign difference between \eqref{TJ1} and \eqref{TJ2} shows that we are skew for mixed terms as well. Therefore $(N,g_N,J)$ is nearly Kähler.
\end{proof}

\begin{corollary}
The locally defined Riemannian submersion $\pi$ gives rise to a globally defined submersion $\pi\colon M\to N$ where $(N,g_N,J)$ is a nearly Kähler orbifold.
\end{corollary}

\begin{proof}
We need to proof that the $\R$-action generated by $\xi$ is locally free. Now $\xi$ generates a $1$-dimensional subgroup of the group $\mathrm{SU}(2)$ generated by $\mathcal{V}$. Therefore the orbits of $\xi$ are compact $S^1$, in particular the action is locally free.
\end{proof}

Coming from parallel $3$-$(\alpha,\delta)$-Sasaki manifolds these nearly Kähler spaces are rather special, inheriting additional properties.

\begin{prop}
The nearly Kähler spaces obtained through \autoref{parallelovernk} have reducible characteristic holonomy.
\end{prop}

\begin{proof}
As in the proof of \autoref{parallelovernk} we may assume that $\xi=\xi_1$. We show that the holonomy representation $TM=\langle\xi_1\rangle\oplus\langle\xi_2,\xi_3\rangle\oplus \mathcal{H}$ satisfies the conditions of \autoref{holbelow}. Since $\nabla$ preserves each Reeb vector field individually the aforementioned decomposition is $\mathrm{Hol}_0(\nabla)$-invariant. By \eqref{liexi} the distribution $\langle\xi_2,\xi_3\rangle$ is invariant under $\xi_1$ and, thus, projectable. As $\xi_1$ is Killing the same is true for $\mathcal{H}$.
\end{proof}

\begin{rem}
This shows that projectability is essential in \autoref{holbelow} as both $\xi_2,\xi_3$ are parallel with respect to $\nabla$ but their projections individually are not.
\end{rem}

Complete strictly nearly Kähler 6-folds with reducible characteristic holonomy were investigated by \cite{BelgunMoroianu}. They show that the only such manifolds are the twistor spaces $\mathbb{C}P^3$ and $F(1,2)$ with their standard nearly Kähler structures. More generally we see that $(N,g_N,J)$ is of special algebraic torsion in the notation of \cite{Nagy}. That is $\check{T}(X,Y)=0$ for $X,Y\in\pi_*\langle\xi_2,\xi_3\rangle$ and $\check{T}(X,Y)\subset \pi_*\langle\xi_2,\xi_3\rangle$ for any $X,Y\in\mathcal{H}$.  Indeed, from \eqref{checkT} the projections of $\check{T}$ to $\Lambda^3\pi_*\langle\xi_2,\xi_3\rangle$, $\Lambda^2\pi_*\langle\xi_2,\xi_3\rangle\wedge\pi_*\mathcal{H}$ and $\Lambda^3\pi_*\mathcal{H}$ all vanish. In \cite{Nagy} the author distinguishes types of nearly Kähler structures with reducible characteristic holonomy $TN=\check{\mathcal{V}}\oplus\check{\mathcal{H}}$ via eigenvalues of the endomorphism $F\colon \check{\mathcal{H}}\to\check{\mathcal{H}}$ defined by
\begin{align*}
F&\coloneqq-\sum(\nabla^{g_N}_{e_i}J)^2
\end{align*}
where the sum is taken over a basis $\{e_i\}$ of $\check{\mathcal{V}}$. In the situation at hand $\check{\mathcal{V}}=\pi_*\langle \xi_2,\xi_3\rangle$, $\check{\mathcal{H}}=\pi_*\mathcal{H}$ and
\begin{align*}
F&=-\sum_{i=2,3}(\nabla^{g_N}_{\xi_i\circ s}J)^2=-\sum_{i=2,3}((\xi_i\circ s)\intprod\check{T})^2=-4\alpha^2(\varphi_2^2+\varphi_3^2)=8\alpha^2\mathrm{id}_\mathcal{H}
\end{align*}
where we used $(\nabla J)J=-J(\nabla J)$ and \eqref{checkT}.

In his classification of nearly Kähler manifolds, \cite{Nagy}, Nagy shows that such a manifold, if complete, is either homogeneous of type $3$ in his notation or the twistor space of a quaternionic Kähler manifold. We prove with canonical submersions a local version similar to \cite[Corollary 7.7]{Alexandrov}.

\begin{theorem}
Let $(N,g_N,J)$ be a nearly Kähler manifold with characteristic connection $\nabla^{T^N}$. 
Assume the tangent space $TN=\check{\mathcal{V}}\oplus\check{\mathcal{H}}$ splits into $\mathrm{Hol}_0(\nabla^{T^N})$- and $J$-invariant subsets such that the characteristic torsion satisfies $T^N\in \Lambda^2\check{\mathcal{H}}\wedge\check{\mathcal{V}}$. Then there is a locally defined Riemannian submersion $\pi\colon N\to \check{N}$ along $\check{\mathcal{V}}$. Furthermore,  if $F=k\mathrm{id}_{\check{\mathcal{H}}}$, $k>0$, and $\dim\check{\mathcal{V}}=2$ then $\check{N}$ admits a quaternionic Kähler structure locally defined by
\begin{gather*}
I_1=\pi_*\circ J\circ s_*,\quad I_2=\sqrt{\frac 2k}\pi_*\circ (JV\intprod T^N)\circ s_*,\quad I_3=\sqrt{\frac 2k}\pi_*\circ (V\intprod T^N)\circ s_*
\end{gather*}
for any section $s\colon \check{N}\to N$ of $\pi$ and vertical vector field $V\in\Gamma\check{\mathcal{V}}$ of norm $1$.
\end{theorem}

\begin{rem}
As $\mathrm{Hol}_0(\nabla^{T^N})\subset\mathrm{U}(n)$, the splitting $TN=\check{\mathcal{V}}^2\oplus\check{\mathcal{H}}^{2n-2}$ into $J$- and $\mathrm{Hol}_0(\nabla^{T^N})$-invariant modules is equivalent to $\mathrm{Hol}_0(\nabla^{T^N})\subseteq\mathrm{U}(1)\times\mathrm{U}(n-1)$ with its standard representation on $TN$. The result in \cite{Alexandrov} used this assumption and the following assumption on the complex type \begin{align*}
T^N\in \Lambda^{2,0}\check{\mathcal{H}}\wedge\Lambda^{1,0}\check{\mathcal{V}}\oplus \Lambda^{0,2}\check{\mathcal{H}}\wedge\Lambda^{0,1}\check{\mathcal{V}}.
\end{align*}
\end{rem}

\begin{proof}
$(N,g_N,\nabla^{T^N})$ satisfies the conditions in \autoref{canonicalsub} so we obtain a locally defined Riemannian submersion $\pi\colon N\to\check{N}$ with totally geodesic fibers tangent to $\mathcal{V}$ such that
\begin{align*}
\nabla^{g_{\check{N}}}_XY=\pi_*(\nabla^{T^N}_{\overline{X}}\overline{Y}).
\end{align*}
We check that $I_1,I_2,I_3$ satisfy the quaternion relations. As in \autoref{parallelovernk} we see immediately $I_1^2=-\mathrm{id}$. Observe that
\begin{align*}
(\nabla^{g_N}_{JV}J)^2=(\nabla^{g_N}_{V}J)J(\nabla^{g_N}_{V}J)J=-(\nabla^{g_N}_{V}J)^2J^2=(\nabla^{g_N}_{V}J)^2=-\frac 12 F=-\frac{k}{2}\mathrm{id}
\end{align*}
since $V, JV$ form an orthonormal base of $\check{\mathcal{V}}$. It follows that 
\begin{align}\label{nablaJ2}
(V\intprod T^N)^2=((\nabla^{g_N}_VJ)J)^2=-\frac{k}{2}\mathrm{id}
\end{align}
and analogously for $(JV\intprod T^N)^2$. Therefore $I_2,I_3$ are almost complex structures on $\check{N}$. The quaternionic relations follow immediately from 
\begin{align}\label{nablagnJ}
J(V\intprod T^N)=J(\nabla^{g_N}_V J)J=-(\nabla^{g_N}_{JV}J)J=-(JV\intprod T^N)
\end{align}
 and \eqref{nablaJ2}.

It remains to show that $\nabla^{g_{\check{N}}}$ preserves the subbundle of $\mathrm{End}(T\check{N})$ generated by $I_1,I_2,I_3$. We proceed as in \autoref{parallelovernk} and set $\hat{X}\coloneqq\overline{X}-s_*X$. Then
\begin{align*}
(\nabla^{g_{\check{N}}}_XI_1)Y&=\nabla^{g_{\check{N}}}_X(I_1Y)-I_1(\nabla^{g_{\check{N}}}_XY)=\pi_*(\nabla^{T^N}_{\overline{X}}\overline{(\pi_*(J(s_*Y)))}-J(s_*(\pi_*(\nabla^{T^N}_{\overline{X}}\overline{Y}))))\\
&=\pi_*(\nabla^{T^N}_{s_*X}(J(s_*Y)))+\nabla^{T^N}_{\hat{X}}\overline{(\pi_*(J(s_*Y)))}-J(\nabla^{T^N}_{s_*X}\overline{Y})-J(s_*(\pi_*(\nabla^{T^N}_{\hat{X}}\overline{Y})))\\
&=\pi_*((\nabla^{T^N}_{s_*X}J)(s_*Y)+(\hat{X}\intprod T^N)(J(s_*Y))-J(\hat{X}\intprod T^N)(s_*Y))\\
&=2\pi_*((J\hat{X}\intprod T^N)(s_*Y))
\end{align*}
where we made use of \autoref{nablavert}. This shows $\nabla^{g_{\check{N}}}I_1\in\langle I_2,I_3\rangle$ since $\hat{X}\in\check{\mathcal{V}}=\langle V,JV\rangle$. We play the game once more for $I_3$. Then the covariant derivative of $I_2$ is computed completely analogously.
\begin{align*}
(\nabla^{g_{\check{N}}}_XI_3)Y&=\nabla^{g_{\check{N}}}_X(I_3Y)-I_3\nabla^{g_{\check{N}}}_XY\\
&=\sqrt{\frac 2k}\pi_*(\nabla^{T^N}_{\overline{X}}\overline{(\pi_*((V\intprod T^N)(s_*Y)))}-(V\intprod T^N)(s_*(\pi_*(\nabla^{T^N}_{\overline{X}}\overline{Y}))))\\
&=\sqrt{\frac 2k}\pi_*\big(\nabla^{T^N}_{s_*X}((V\intprod T^N)(s_*Y))+\nabla^{T^N}_{\hat{X}}\overline{(\pi_*((V\intprod T^N)(s_*Y)))}\\
&\qquad\qquad\quad-(V\intprod T^N)(\nabla^{T^N}_{s_*X}\overline{Y})-(V\intprod T^N)(s_*(\pi_*(\nabla^{T^N}_{\hat{X}}\overline{Y})))\big)\\
&=\sqrt{\frac 2k}\pi_*\big((\nabla^{T^N}_{s_*X}(V\intprod T^N))(s_*Y)+(\hat{X}\intprod T^N)(V\intprod T^N)(s_*Y)\\
&\qquad\qquad\quad-(V\intprod T^N)(\hat{X}\intprod T^N)(s_*Y)\big)\\
&=\sqrt{\frac 2k}\pi_*(((\nabla^{T^N}_{s_*X}V)\intprod T^N)(s_*Y))+\sqrt{2k}g_N(JV,\hat{X})I_1Y
\end{align*}
where in the last step we have used \eqref{nablaJ2} and \eqref{nablagnJ} to conclude
\begin{align*}
(\hat{X}\intprod T^N)(V\intprod T^N)-&(V\intprod T^N)(\hat{X}\intprod T^N)\\
&=g(V,\hat{X})((V\intprod T^N)^2-(V\intprod T^N)^2)\\
&\quad+g(JV,\hat{X})((JV\intprod T^N)(V\intprod T^N)-(V\intprod T^N)(JV\intprod T^N))\\
&=-g(JV,\hat{X})(J(V\intprod T^N)^2+(V\intprod T^N)^2J)\\
&=kg(JV,\hat{X})J.
\end{align*}
Now the result follows as $\nabla^{T^N}_{s_*X}V\in\check{\mathcal{V}}=\langle V,JV\rangle$.
\end{proof}	

%

\textbf{Acknowledgments:} I would like to thank Giovanni Russo for many helpful discussions on the topic.

\textbf{Funding information:} The author states no funding involved.

\textbf{Conflict of interest:} The author states no conflicting interests.

\textbf{Author contributions:} The author confirms sole responsibility for the material and manuscript.

\printbibliography

@article{Ag06,
    AUTHOR = {Agricola, Ilka},
     TITLE = {The {S}rn\'i lectures on non-integrable geometries with torsion},
   JOURNAL = {Arch. Math. (Brno)},
  FJOURNAL = {Universitatis Masarykianae Brunensis. Facultas Scientiarum
              Naturalium. Archivum Mathematicum},
    VOLUME = {42},
      YEAR = {2006},
    NUMBER = {suppl.},
     PAGES = {5--84},
      ISSN = {0044-8753},
   MRCLASS = {53C10 (53C25 53C27 53C29 53D15 58J60 81T30)},
  MRNUMBER = {2322400},
MRREVIEWER = {Anna M. Fino},
}

@article {AgrDil,
    AUTHOR = {Agricola, Ilka and Dileo, Giulia},
     TITLE = {Generalizations of 3-{S}asakian manifolds and skew torsion},
   JOURNAL = {Adv. Geom.},
  FJOURNAL = {Advances in Geometry},
    VOLUME = {20},
      YEAR = {2020},
    NUMBER = {3},
     PAGES = {331--374},
      ISSN = {1615-715X},
   MRCLASS = {53C15 (22E25 32V05 53B05 53C25 53C27 53D10)},
  MRNUMBER = {4121338},
       DOI = {10.1515/advgeom-2018-0036},
       URL = {https://doi.org/10.1515/advgeom-2018-0036},
}

@Article{ADS20,
 Author = {Ilka {Agricola} and Giulia {Dileo} and Leander {Stecker}},
 Title = {{Homogeneous non-degenerate \(3- (\alpha,\delta)\)-Sasaki manifolds and submersions over quaternionic K\"ahler spaces}},
 FJournal = {{Annals of Global Analysis and Geometry}},
 Journal = {{Ann. Global Anal. Geom.}},
 ISSN = {0232-704X},
 Volume = {60},
 Number = {1},
 Pages = {111--141},
 Year = {2021},
 Publisher = {Springer Netherlands, Dordrecht},
 Language = {English},
 DOI = {10.1007/s10455-021-09762-9},
 MSC2010 = {53B05 53C15 53C25 53C26 53D10 53C27 32V05 22E25}
}

@Article{Alexandrov,
 Author = {Alexandrov, Bogdan},
 Title = {{{\(Sp(n)U(1)\)}}-connections with parallel totally skew-symmetric torsion},
 FJournal = {Journal of Geometry and Physics},
 Journal = {J. Geom. Phys.},
 ISSN = {0393-0440},
 Volume = {57},
 Number = {1},
 Pages = {323--337},
 Year = {2006},
 Language = {English},
 DOI = {10.1016/j.geomphys.2006.03.005},
}

@article{BelgunMoroianu,
 Author = {Belgun, Florin and Moroianu, Andrei},
 Title = {Nearly {K{\"a}hler} 6-manifolds with reduced holonomy},
 FJournal = {Annals of Global Analysis and Geometry},
 Journal = {Ann. Global Anal. Geom.},
 Volume = {19},
 Number = {4},
 Pages = {307--319},
 Year = {2001},
 Language = {English},
 DOI = {10.1023/A:1010799215310},
}

@Article{Bry,
 Author = {Bryant, Robert L.},
 Title = {Metrics with exceptional holonomy},
 FJournal = {Annals of Mathematics. Second Series},
 Journal = {Ann. Math. (2)},
 ISSN = {0003-486X},
 Volume = {126},
 Pages = {525--576},
 Year = {1987},
 Language = {English},
 DOI = {10.2307/1971360},
 zbMATH = {4038683},
 Zbl = {0637.53042}
}

@Article{CleyMorSemm,
 Author = {Cleyton, Richard and Moroianu, Andrei and Semmelmann, Uwe},
 Title = {Metric connections with parallel skew-symmetric torsion},
 FJournal = {Advances in Mathematics},
 Journal = {Adv. Math.},
 ISSN = {0001-8708},
 Volume = {378},
 Pages = {51},
 Note = {Id/No 107519},
 Year = {2021},
 Language = {English},
 DOI = {10.1016/j.aim.2020.107519},
}

@Article{EellsSalamon,
 Author = {Eells, J. and Salamon, S.},
 Title = {Twistorial construction of harmonic maps of surfaces into four-manifolds},
 FJournal = {Annali della Scuola Normale Superiore di Pisa. Classe di Scienze. Serie IV},
 Journal = {Ann. Sc. Norm. Super. Pisa, Cl. Sci., IV. Ser.},
 ISSN = {0391-173X},
 Volume = {12},
 Pages = {589--640},
 Year = {1985},
 Language = {English},
 zbMATH = {4020045},
 Zbl = {0627.58019}
}

@Article{FriGru,
 Author = {Friedrich, Th. and Grunewald, R.},
 Title = {On {Einstein} metrics on the twistor space of a four-dimensional {Riemannian} manifold},
 FJournal = {Mathematische Nachrichten},
 Journal = {Math. Nachr.},
 ISSN = {0025-584X},
 Volume = {123},
 Pages = {55--60},
 Year = {1985},
 Language = {English},
 DOI = {10.1002/mana.19851230106},
 zbMATH = {3913203},
 Zbl = {0572.53037}
}

@article {FrIv,
    AUTHOR = {Friedrich, Thomas and Ivanov, Stefan},
     TITLE = {Parallel spinors and connections with skew-symmetric torsion
              in string theory},
   JOURNAL = {Asian J. Math.},
  FJOURNAL = {Asian Journal of Mathematics},
    VOLUME = {6},
      YEAR = {2002},
    NUMBER = {2},
     PAGES = {303--335},
      ISSN = {1093-6106},
   MRCLASS = {53C27 (53C05 53C25 58J60)},
  MRNUMBER = {1928632},
MRREVIEWER = {Uwe Semmelmann},
       DOI = {10.4310/AJM.2002.v6.n2.a5},
       URL = {https://doi.org/10.4310/AJM.2002.v6.n2.a5},
}

@Article{FriWeakSpin9,
 Author = {Friedrich, Thomas},
 Title = {Weak {Spin}(9)-structures on 16-dimensional {Riemannian} manifolds},
 FJournal = {The Asian Journal of Mathematics},
 Journal = {Asian J. Math.},
 ISSN = {1093-6106},
 Volume = {5},
 Number = {1},
 Pages = {129--160},
 Year = {2001},
 Language = {English},
 DOI = {10.4310/AJM.2001.v5.n1.a9},
 zbMATH = {1818528},
 Zbl = {1021.53028}
}

@Article{FriSpin9Skew,
 Author = {Friedrich, Thomas},
 Title = {Spin(9)-structures and connections with totally skew-symmetric torsion.},
 FJournal = {Journal of Geometry and Physics},
 Journal = {J. Geom. Phys.},
 ISSN = {0393-0440},
 Volume = {47},
 Number = {2-3},
 Pages = {197--206},
 Year = {2003},
 Language = {English},
 DOI = {10.1016/S0393-0440(02)00189-4},
 zbMATH = {1963591},
 Zbl = {1039.53049}
}

@article{GMWSkew,
  title = {Superstrings with intrinsic torsion},
  author = {Gauntlett, Jerome P. and Martelli, Dario and Waldram, Daniel},
  journal = {Phys. Rev. D},
  volume = {69},
  issue = {8},
  pages = {086002},
  numpages = {27},
  year = {2004},
  publisher = {American Physical Society},
  doi = {10.1103/PhysRevD.69.086002},
  url = {https://link.aps.org/doi/10.1103/PhysRevD.69.086002}
}

@Article{GraynK,
 Author = {Alfred {Gray}},
 Title = {{The structure of nearly K\"ahler manifolds}},
 FJournal = {{Mathematische Annalen}},
 Journal = {{Math. Ann.}},
 ISSN = {0025-5831; 1432-1807/e},
 Volume = {223},
 Pages = {233--248},
 Year = {1976},
 Publisher = {Springer, Berlin/Heidelberg},
 Language = {English},
 MSC2010 = {53C15 53B20 57R20},
 Zbl = {0345.53019}
}

@Article{GrayWeak,
 Author = {Gray, Alfred},
 Title = {Weak holonomy groups},
 FJournal = {Mathematische Zeitschrift},
 Journal = {Math. Z.},
 ISSN = {0025-5874},
 Volume = {123},
 Pages = {290--300},
 Year = {1971},
 Language = {English},
 DOI = {10.1007/BF01109983},
 zbMATH = {3352418},
 Zbl = {0222.53043}
}

@article{MorSch,
      title={Submersion constructions for geometries with parallel skew torsion}, 
      author={Andrei Moroianu and Paul Schwahn},
      year={2024},
      eprint={2409.14421},
      archivePrefix={arXiv},
      primaryClass={math.DG},
      url={https://arxiv.org/abs/2409.14421}, 
}

@Article{Nagy,
 Author = {Nagy, Paul-Andi},
 Title = {Nearly {K{\"a}hler} geometry and {Riemannian} foliations.},
 FJournal = {The Asian Journal of Mathematics},
 Journal = {Asian J. Math.},
 ISSN = {1093-6106},
 Volume = {6},
 Number = {3},
 Pages = {481--504},
 Year = {2002},
 Language = {English},
 DOI = {10.4310/AJM.2002.v6.n3.a5},
}

@phdthesis{mythesis,
	Author = {Leander {Stecker}},
	Title = {On $3$-$(\alpha,\delta)$-Sasaki manifolds and their canonical submersions},
	school = "Philipps-Universität",
	adress = "Marburg",
	year = 2021,
	month = aug
}
\textsc{Leander Stecker, Universität Leipzig,\\ Augustusplatz 10, 04019 Leipzig, Germany} \\
\textit{E-mail address}: \texttt{leander.stecker@uni-leipzig.de}

\end{document}